\documentclass[11pt,naturalnames]{article}
\usepackage{amsmath}
\usepackage{amssymb}
\usepackage{amsthm}
\usepackage{amsfonts}
\usepackage{color}
\usepackage[hyphens]{url}
\usepackage{hyperref}
\usepackage{graphicx}
\usepackage{color}
\usepackage{tabularx}
\usepackage{subfig}
\usepackage[T1]{fontenc}

\usepackage[titletoc,title]{appendix}

\usepackage{cite}
\usepackage{bbm}

\def\Var{\text{Var}}

\def\bmin{\beta_{\mathrm{min}}}
\def\bmax{\beta_{\mathrm{max}}}
\def\vmin{v_{\mathrm{min}}}
\def\vmax{v_{\mathrm{max}}}

\theoremstyle{plain}
\newtheorem{thm}{Theorem}[section]

\newtheorem{prop}[thm]{Proposition}
\newtheorem{prob}[thm]{Problem}

\title{Best and worst policy control in low-prevalence SEIR}

\date{  }
\author{
\begin{tabular}{c} Scott Sheffield\\[-5pt]\small MIT \end{tabular}
}

\begin{document}

\maketitle

\begin{abstract}
We consider the low-prevalence linearized SEIR epidemic model for a society that has resolved to keep future infections low in anticipation of a vaccine. The society can vary its amount of potentially-infection-spreading activity over time, within a certain feasible range. Because the activity has social or economic value, the society aims to maximize activity overall subject to infection rate constraints.

We find that consistent policies are the {\em worst possible} in terms of activity, while the best policies alternate between high and low activity. In a variant involving multiple subpopulations, we find that the best policies are maximally {\em coordinated} (maintaining similar prevalence among subpopulations) but {\em oscillatory} (having growth rates that vary in time).

It turns out that linearized SEIR is mathematically equivalent to an idealized racecar model (with different subpopulations corresponding to different cars) and the amount of fuel used corresponds to the amount of activity. Using this analogy, steady V-shaped formations (in which one subpopulation ``leads the way'' with consistently higher prevalence and activity, while others follow behind with lower prevalence and activity) are especially problematic. These formations are {\em very effective} at minimizing fuel use, hence {\em very ineffective} at boosting activity.  In an appendix, we obtain analogous results for alternative notions of activity, which incorporate crowding effects.
\end{abstract}

\section{Introduction}

Non-pharmaceutical interventions (NPIs), including the prohibition or rescheduling of activities, are used to limit the damage caused by pandemics. When managing a pandemic, a government or population may decide that the cost of obtaining significant herd immunity through infection is unacceptably high, and that it is therefore necessary to maintain low disease prevalence through NPIs until vaccines are available.

One important and controversial question is the following: once disease prevalence is low (say, 50 confirmed cases per million per day), is it better to keep the effective reproductive rate as close to one as possible (aiming for consistency and sustainability) or to ``suppress and resuppress'' (i.e., drive the prevalence down even further, then relax restrictions, then restore restrictions if/when prevalence rates recover, etc.)?

In a recent paper, the author and others showed that strict but intermittent measures were better than consistently moderate measures at optimizing certain utility functions within the low-prevalence limit of SEIR/SEIS and its variants  \cite{lockdownscount2020}. In this follow-up note, we show that consistent strategies are actually the {\em least effective} when measured by a certain type of {\em activity}. We also work out the {\em most effective} strategies, including in settings where upper and lower bounds on the infection rate are imposed. We then explore settings with multiple subpopulations and find that, in terms of prevalence, it is much better to aim for {\em geographic consistency} and {\em temporal variability} than other way around.

The models in this paper are simplifications that omit important considerations. They are meant to generate hypotheses and inform judgment, not resolve real world questions on their own. Nonetheless, it is interesting that, within these models, some of the {\em intuitively-best-sounding} practices are actually the worst.

{\bf Acknowledgement:}
We thank Morris Ang, Minjae Park, Joshua Pfeffer, Pu Yu, and the co-authors of \cite{lockdownscount2020} for useful conversations. The author is partially supported by NSF award DMS 1712862.

\section{Methods}

\subsection{Setup and motivation: a linearized SEIR optimal control problem}
During a pandemic, there may come a time when a state or country resolves to keep its total number of future infections small (perhaps less than $2$ percent of the population) up until a later time $T$ at which a vaccination program will commence.  If we assume dynamics are given by a standard SEIR model, this means that $S$ and $R$ can be treated as (essentially) constant for the remaining duration. We choose our time unit so that the mean incubation time is $1$, and the mean infectious time is the constant $\gamma^{-1}$. We then obtain the {\em linearized} ODE\cite{diekmann2010construction} given by \begin{equation} \label{eqn::note} \dot E(t) = -E(t) + \beta(t) I(t), \,\,\,\,\,\,\,\,  \,\,\,\,\, \dot I(t) = E(t) - \gamma I(t),\end{equation} where $E(t)$ represents the fraction of the population {\em exposed} (infected but not yet infectious), $I(t)$ represents the fraction that is {\em infectious}, and $\beta(t) \in [\bmin,\bmax]$ is a control parameter (a measurable function of $t$) describing the rate at which disease transmission occurs at time $t$.

We assume that there is some {\em flexible activity} (haircuts, conversations, surgeries, lessons, factory shifts, etc.)\ that has social/economic value but also carries transmission risk. By ``flexible'' we mean that its utility is not dependent on {\em when} it occurs. Our policy tool is deciding how much of this activity to schedule/allow and when to do so. For now, we assume that the disease transmission caused by flexible activity is primarily due to the activity itself (not ancillary crowding effects) so that disease transmission is linear in the amount of activity. (We will discuss alternatives in Appendix~\ref{app:generalutility}.) We interpret $\beta(t)$ as the {\em total} amount of transmission-inducing activity (flexible or otherwise) happening at time $t$.  We interpret $\bmin$ as the amount of transmission-inducing activity that occurs when no flexible activity is scheduled, so that $\beta(t) - \bmin$ is the amount of {\em flexible} activity at time $t$. We interpret $\bmax-\bmin$ as the maximal amount of flexible activity that can be scheduled at once (due to limitations on space or on the number of individuals available to be active).   Define the {\em total activity} by $A = \int_0^T \beta(t) dt$. Our goal will be to find strategies that (subject to restrictions) maximize $A$.

The main assumption underlying this goal is that the social/economic utility derived from flexible activity depends only on $A$, not on how the flexible activity is temporally distributed.  We {\em do not} assume that all flexible activity is of equal value. For example, we allow for the possibility that if $A$ were small, only very important activity would be allowed, but if $A$ were larger, more discretionary activity would take place.\footnote{For example, if a maskless conversation contributes twice as much disease transmission risk as a similar masked conversation, then the maskless conversation would count as twice as much ``activity.'' So ``removing a mask during a close conversation'' could be treated as a form of discretionary activity that might only occur in larger $A$ scenarios. If an 8-person party involves more than twice as much transmission as a 4-person party, then it would count as more than twice as much activity.} As long as social/economic utility is an increasing function of $A$ ({\em not necessarily linear}) it is sensible for maximizing $A$ to be an objective.

Alternative objective functions (accounting for activity that is {\em not} perfectly flexible, e.g.\ because people pursuing different activities at once might {\em crowd} each other in a way that increases transmission) will be discussed in Appendix~\ref{app:generalutility}. For now, we will focus only on maximizing $A$, not on minimizing total infections. (Effectively, we are assuming that prevalence is low enough that further minimizing infections is not the primary consideration.) But we will consider imposing upper and lower bounds on infection rates. See also \cite{lockdownscount2020} for further references, as well as some discussion of the probability distributions governing incubation and infectious periods, social networks, and other factors beyond the scope of this note.

We assume a vaccine will arrive at time $T$. We do not model the vaccination strategy, but we allow that the cost of controlling the disease during its implementation may depend on the terminal values $E(T)$ and $I(T)$. Thus it is of interest to find the optimal $\beta:[0,T] \to [\bmin,\bmax]$ yielding any particular choice of $E(T)$ and $I(T)$, and to find the amount of activity that approach yields. (Finding the optimal values for $E(T)$ and $I(T)$ would then be a second step, and would depend on the vaccination rollout model used.)

\subsection{Problem statements}

\begin{prob} \label{prob::activity}
Given $E(0)$, $I(0)$, $E(T)$ and $I(T)$, find the $\beta$ that maximizes $A$.
\end{prob}

To simplify the presentation, we change coordinates to reduce Problem~\ref{prob::activity} to a one-dimensional problem. Define the {\em velocity} of the disease to be $V(t) := E(t)/I(t)$. The term ``velocity'' is motivated by the fact that
\begin{equation}\label{eqn::logI} \frac{\partial}{\partial t} \log I(t) = \frac{\dot I(t)}{I(t)} = V(t) - \gamma, \,\,\,\,\,\,\,\,\,\,  \log\frac{I(T)}{I(0)} = \int_0^T V(t) dt - \gamma T, \end{equation}
so that $V(t)$ is (up to additive constant) the rate at which $\log I(t)$ is changing. Computing further we find

\begin{align} \label{eqn::Y} \dot V(t) &= \frac{ \dot E(t) I(t) - \dot I(t)E(t)}{I(t)^2} \\
&=  \frac{ \bigl(- E(t) + \beta(t)I(t)\bigr) I(t) - \bigl(E(t) - \gamma I(t) \bigr) E(t)}{I(t)^2}  \nonumber \\
&=  \frac{  -E(t)I(t) + \beta(t) I(t)^2 - E(t)^2 + \gamma I(t)E(t)}{I(t)^2}    \nonumber \\
&= \beta(t) -V(t)^2 + (\gamma-1)V(t) \nonumber \\
&= \beta(t) -\phi(V(t)), \nonumber
\end{align}
where $\phi$ is the quadratic function defined by $\phi(x) := x^2 + (1-\gamma)x$. This also implies $\beta(t) = \dot V(t) + \phi(V(t))$. We then have
\begin{equation} \label{eqn::A}
A = \int_0^T \phi\bigl(V(t)\bigr)dt + V(T)-V(0) = \int_0^T  V(t)^2 dt  +  \int_0^T (1-\gamma) V(t) dt + V(T) - V(0).
\end{equation}

The existence and uniqueness of solutions to \eqref{eqn::Y} are immediate from the Carath\'eodory existence theorem and the corresponding uniqueness conditions, provided $\beta$ is a measurable function from $[0,T]$ to $[\bmin,\bmax]$.  See \cite[Theorem 5.3]{hale1980ODE}. If the reader prefers not to consider this much generality, it is fine to focus on the case that $\beta$ is piecewise continuous (or even piecewise constant).

\begin{figure}[ht!]
\begin{center}
\includegraphics[scale=0.4]{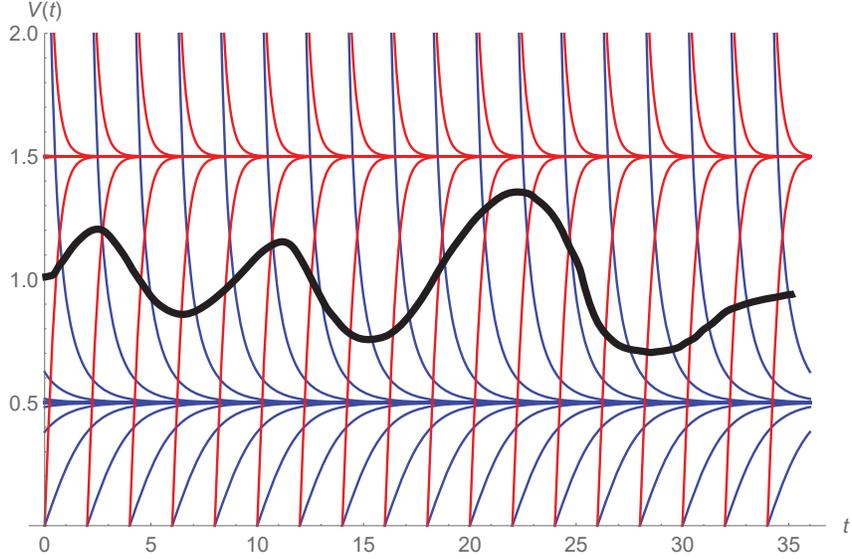}
\caption{\label{fig::ychart} {\bf Blue:} solutions to $\dot V(t) = \beta(t) - \phi(V(t))$ with $\beta(t) = \bmin=1/4$. Here $\gamma=1$ so $\phi(x)=x^2$. {\bf Red:} solutions with $\beta(t)=\bmax=9/4$. {\bf Black:} possible solution with $\beta(t)$ varying within $[\bmin,\bmax]$. Regardless of the initial $V(0)>0$, the (blue) solutions with $\beta(t)=\bmin$ for all time (no flexible activity) converges to $\vmin = \phi^{-1}(\bmin)=1/2$.  Similarly, the (red) solutions with $\beta(t) = \bmax$ for all time (maximal flexible activity) converge to $\vmax = \phi^{-1}(\bmax)=3/2$.  Generally, the $V(t)$ satisfying \eqref{eqn::Yconstraint} are Lipschitz continuous curves such that $\dot V(t)$ (which is a.e.\ defined) always lies {\em between} the derivatives of the blue and red curves that pass through $(t,V(t))$.  Other examples include {\bf bang-bang alternators}, where $V(t)$ alternates between tracing red and blue curves as in Figure~\ref{fig::examplechart} and {\bf constant functions} $V(t) = c$ for $c \in [\vmin,\vmax]$.  The {\bf special constant function} $V(t) = \gamma = 1$ is the curve for which $E(t)$, $I(t)$, and $\beta(t)=1$ all remain constant. The time unit is the mean incubation period; if mean incubation is $4$ days, then the $36$ units above represent 144 days.
}
\end{center}
\end{figure}


Now note that \eqref{eqn::logI} and the definition of $V$ imply that fixing the quadruple $\bigl( E(0), I(0), E(T), I(T) \bigr)$ modulo multiplicative constant is equivalent to fixing the triple $\bigl( V(0), V(T), \int_0^T V(t)dt \bigr)$.  So we rephrase Problem~\ref{prob::activity} as an equivalent problem of maximizing \eqref{eqn::A} given this triple. Since the latter three RHS terms are determined by the triple, this is equivalent to maximizing the first term $\int_0^T V(t)^2 dt$, hence equivalent to the following:

\begin{prob} \label{prob::activity2}
Given prescribed values for $V(0)$, $V(T)$, and $\int_0^T V(t) dt$, find a Lipschitz function $V$ that maximizes $\int_0^T V(t)^2dt$ subject to the constraint $\dot V(t) + \phi\bigl(V(t)\bigr)  \in [\bmin, \bmax]$ for all $t$, or equivalently
\begin{equation} \label{eqn::Yconstraint} \bmin - \phi(V(t)) \leq \dot V(t) \leq \bmax - \phi(V(t)).
\end{equation} Alternative phrasing: let $t$ be a random variable chosen uniformly from $[0,T]$ and choose $V$ to maximize the variance of $V(t)$ given \eqref{eqn::Yconstraint} and a prescribed value for the expectation of $V(t)$, as well as $V(0)$ and $V(T)$.
\end{prob}

See Figure~\ref{fig::ychart} for an intuitive picture of what the $V(t)$ satisying \eqref{eqn::Yconstraint} are like. For any $b > 0$ we write $\phi^{-1}(b)$ for the unique positive solution $x$ to $\phi(x) = b$.  By the quadratic formula, $\phi^{-1}(b) = \frac{(\gamma-1) + \sqrt{(\gamma-1)^2 + 4 b}}{2}$.  It is easily seen from \eqref{eqn::Y} that if we set $\beta(t) = b$ for all $t$, then $V(t)$ converges to $\phi^{-1}(b)$ as $t \to \infty$ (regardless of the initial value $V(0)$).  For short, write $\vmin=\phi^{-1}(\bmin)$ and $\vmax = \phi^{-1}(\bmax)$.  Note that if $V(0) \in [\vmin, \vmax]$ then $V(t) \in [\vmin, \vmax]$ for all time $t$, regardless of $\beta$.  

We also consider a constrained version:

\begin{prob} \label{prob::activity3}
Solve Problem~\ref{prob::activity} with the added constraint that $C_1 \leq \log I(t) \leq C_2$ for all $t \in [0,T]$.
\end{prob}

As motivation, note that imposing an upper bound on $\log I(t)$ is a way to ensure that a health care system is not overwhelmed and to limit the daily risk assumed by individual workers. It is also a crude way to ensure that the overall number of infections does not become too large: perhaps $I(t) = e^{C_2}$ is about the level at which the price of infection becomes unacceptable.  On the other side, if neighboring states and countries have not eliminated the disease, and are maintaining steady levels of infection, then cases may be reintroduced from those localities at some small but steady rate, which would effectively impose a lower bound on $\log I(t)$.\footnote{The upward drift on $\log I(t)$ caused by this influx, which is more pronounced when $\log I(t)$ is low, might be offset to some degree by contract tracing that is more effective when $\log I(t)$ is low. Other low-prevalence considerations (randomness, possible periods with no disease, large jumps due to superspreaders, etc.) are beyond the scope of this note.
An alternative to the rigid lower bound is to add an extra fixed-prevalence subpopulation, in the language of Section~\ref{sec::subpopulations}, that is only weakly connected to the other subpopulations. An alternative to the rigid upper bound is to subtract a multiple of $\int_0^T I(t)dt$ from the objective function, which would heavily penalize larger $\log I(t)$ values but would not matter much for smaller $\log I(t)$ values.}

\subsection{Physics analogy} \label{subsec::physics}

As an instructive metaphor, interpret $X(t) := \log I(t)+\gamma t$ as the {\em position} of a rocket-powered car along a frictionless street, the derivative $\dot X(t) = V(t)$ as the {\em velocity}, and the second derivative $\dot V(t) = \beta(t) - \phi(V(t))$ as the {\em acceleration}.  Interpret $\beta(t)$ as an {\em internal force} applied via the gas pedal and $-\phi(V(t))$ as an {\em external force} which is a quadratic function of the velocity, accounting for wind resistance and/or gravity.\footnote{If $\gamma=1$, then $-\phi(V(t)) = -V(t)^2$ is the standard {\em quadratic drag} used to model wind resistance. If $\gamma \not =1$ then $\phi(V(t)) = \bigl( V(t) - \frac{\gamma-1}{2}\bigr)^2 - \bigl( \frac{\gamma-1}{2}\bigr)^2$, which corresponds to a prevailing wind of speed $\frac{\gamma-1}{2}$ and a street sloped to yield a velocity-independent gravitational force of $\bigl( \frac{\gamma-1}{2}\bigr)^2$. The metaphor breaks down if $V(t) < \frac{\gamma-1}{2}$ (i.e., if the windspeed is forward but the car is moving slower than the wind) since in this case the force from the wind is in the wrong direction.}  Here $\bmin$ corresponds to the gas pedal not being pressed and $\bmax$ corresponds to a fully pressed pedal (there are no brakes). We can interpret $A$ as the total amount of fuel used\footnote{Assume fuel weight is small compared to car weight, so overall car weight does not change.} and imagine we are trying to waste as much fuel as possible subject to given boundary conditions.  When we impose the constraints $C_1 \leq \log I(t) \leq C_2$ we interpret them as bounding the car between two trucks moving at constant speed. See Figures~\ref{fig::trucks} and~\ref{fig::examplechart}.

\begin{figure}[ht!]
\begin{center}
\includegraphics[scale=0.9]{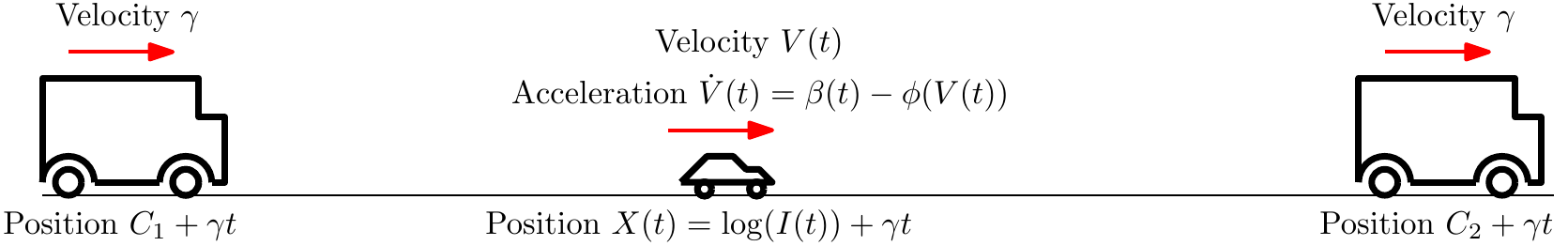}
\caption{\label{fig::trucks} The objective is to maximize $\int \beta(t)dt$ given the car's initial and final position and velocity. Because the external force $-\phi(V(t))$ is quadratic in the velocity $V(t)$ (and the mean velocity and acceleration are determined by the initial and final position and velocity) this is equivalent to maximizing $\int \phi(V(t))dt$, which in turn is equivalent to maximizing the variance of the car's velocity. Note that in Figure~\ref{fig::ychart}, a red curve started at height $\vmin =.5$ gets almost to $\vmax=1.5$ quickly (over a couple time units, i.e.\ a couple multiples of the mean incubation time); and a blue curve started at $\vmax=1.5$ gets almost to $\vmin=.5$ quickly. This means the car can quickly change speeds from (roughly) $\vmin$ to (roughly) $\vmax$ and back. It makes intuitive sense that, in order to maximize the car's velocity variance over a long time period, one might want to alternate between $\beta(t)=\bmax$ (holding down the pedal until the car is close to the leading truck) and $\beta(t)=\bmin$ (releasing the pedal until the car is close to the trailing truck) as in Figure~\ref{fig::examplechart}.}
\end{center}
\end{figure}

\begin{figure}[ht!]
\begin{center}
\includegraphics[scale=0.24
]{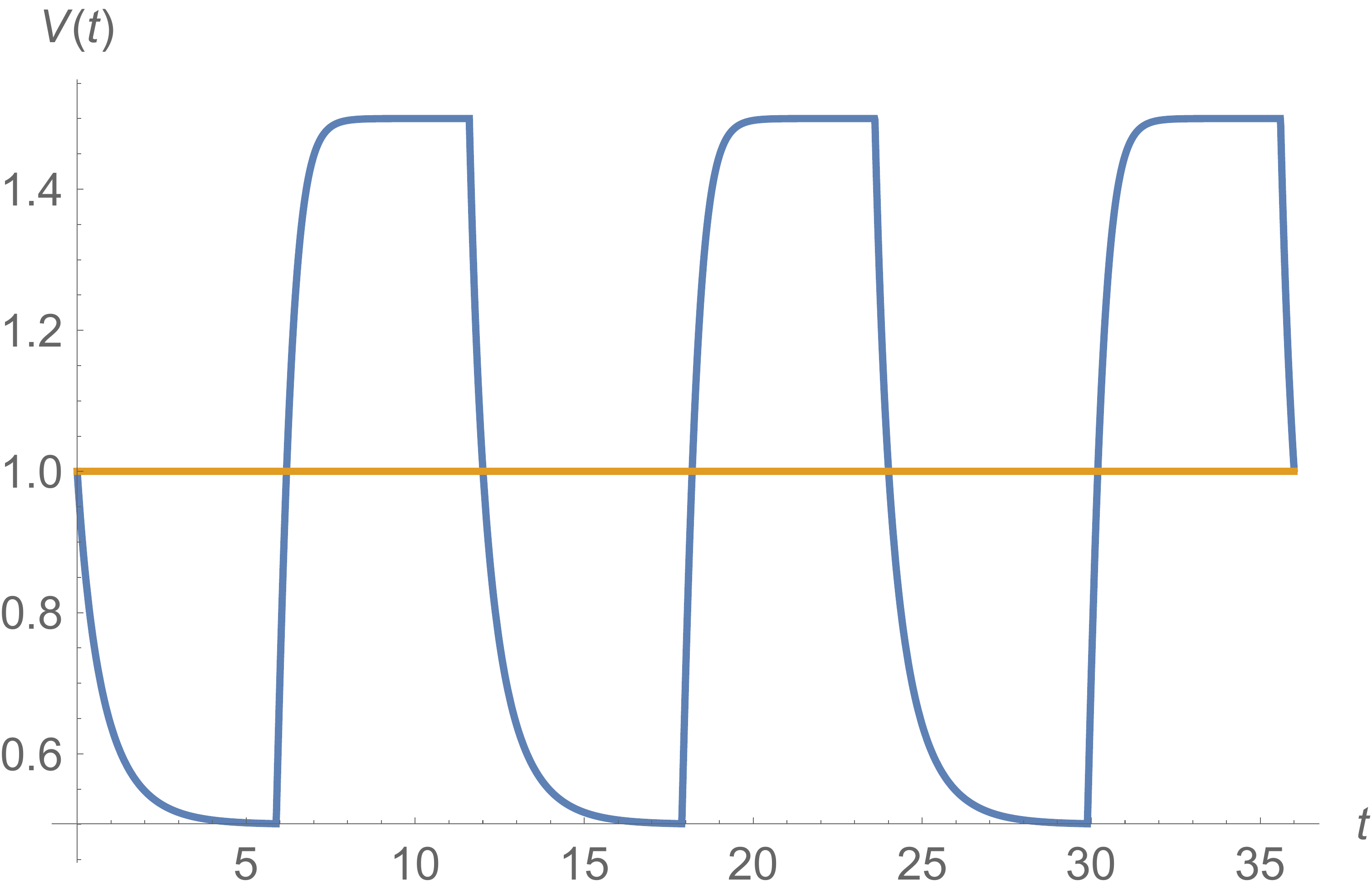}\hspace{.5in} \includegraphics[scale=0.25]{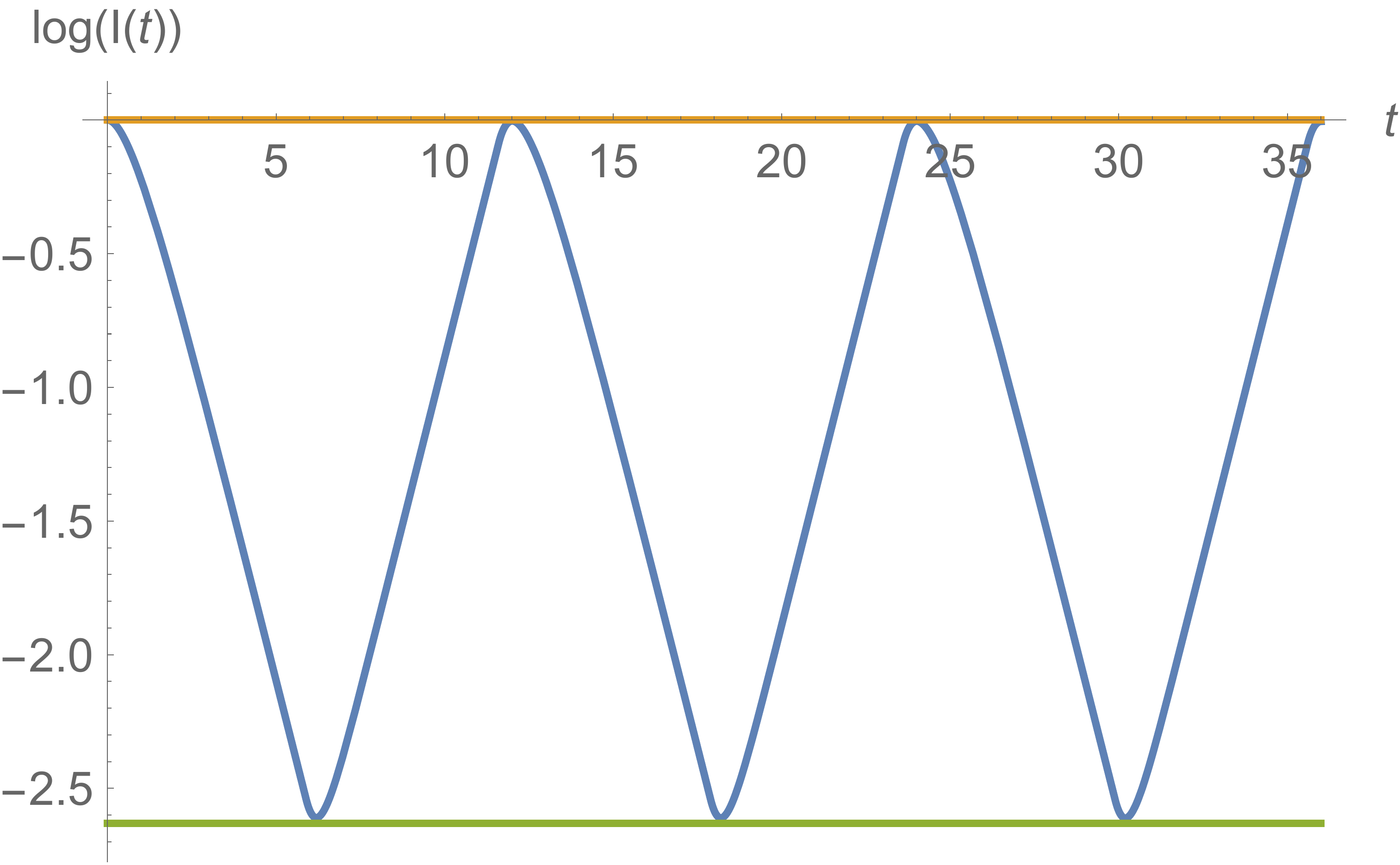}
\caption{\label{fig::examplechart} Example $V(t)$ in the setting of Figure~\ref{fig::ychart} (left) and the corresponding $\log I(t)$ (right) obtained by integrating $V(t)-\gamma$.  Here $\gamma=1$, $V(0) = \gamma$ and the function $V(t)$ makes six excursions away from $\gamma$ (three below and three above) which correspond to the six alternating intervals on which $\log I(t)$ decreases or increases. In this example, $\beta$ alternates between $\bmin$ and $\bmax$ and $\log I(t)$ oscillates between $C_2=0$ (orange line) and $C_1 \approx -2.6$ (green line) so that $I(t)$ changes by a factor of about $e^{2.6} \approx 13.5$. Each 12-unit period has $5.712$ units of $\beta(t)=2.25$ and $6.288$ units of $\beta(t)=.25$. The activity per 12-unit period is $6.288\cdot .25 +5.712 \cdot 2.25=14.424$ which is 20 percent higher than $12$. That is, the alternating $\beta$ scenario (blue curve) allows 20 percent more activity than the constant $\beta(t)=1$ scenario (orange curve).}
\end{center}
\end{figure}

\section{Results}
\subsection{Activity minimizers and maximizers}
The following is immediate from the statement of Problem~\ref{prob::activity2} and the fact that the variance of a constant random variable is zero.

\begin{prop} \label{prop::worstpossible}
In the setting of Problem~\ref{prob::activity2}, if one fixes $V(0) = V(T)  \in [\vmin,\vmax]$ and sets $\int_0^T V(t)dt = T\cdot V(0)$ then the {\em minimal} activity solution is the constant-velocity solution with $V(t) = V(0)$ for all $t$, which corresponds to $\beta(t) = \phi\bigl( V(0)\bigr)$ for all $t$.  In other words, if one aims to maximize $A$, constant velocity strategies are the {\em worst possible}.
\end{prop}

If $V(0) \not = V(T)$, then it is not possible to make the velocity variance exactly zero; but the worst possible strategy is still the one that makes this variance as small as possible. Glancing at Figure~\ref{fig::ychart},  it is intuitively clear that if one wanted to {\em maximize} the variance of $V(t)$, given its mean and its initial and final values, then one would ideally want $V(t)$ to spend most of its time near $\vmin=.5$ or near $\vmax=1.5$, with as little time as possible spent transitioning between the two sides. One might guess that if $V(0) \not= V(T)$ the optimal strategy would be this: first move $V(t)$ as quickly as possible toward one side of $[\vmin,\vmax]$ (the one accessible without crossing $V(T)$), then at some point move $V(t)$ as quickly as possible toward the other side, and then at some point move $V(t)$ as quickly as possible toward $V(T)$. This is correct and we formalize this as follows.

\begin{prop} \label{prop::bestpossible}
In the setting of Problem~\ref{prob::activity2}, if $V(T) > V(0)$, then any {\em optimal solution} has $\beta(t) = \bmax$ on a single interval, with $\beta(t) = \bmin$ before and after that.  If $V(T) < V(0)$, then the optimal solution has $\beta(t) = \bmin$ on a single interval, with $\beta(t) = \bmax$ before and after that.  If $V(T) = V(0)$ then there is an optimal solution of each of the two types mentioned above.
\end{prop}

This is proved by showing that if $V$ does {\em not} have the form described then one can modify it in a way that increases $\int V(t)^2dt$ while keeping $\int V(t)dt$ the same. See Appendix~\ref{app::proofs} for details. A similar argument is made for more general objective functions in Appendix~\ref{app:generalutility}.

\begin{prop}\label{prop::wallconstrained} In the setting of Problem~\ref{prob::activity3},  with $\gamma \in (\vmin,\vmax)$, assume $\log I(0)$ and $\log I(T)$ are fixed values in $\{C_1,C_2 \}$ and $E(0)$ and $E(T)$ are fixed so that $V(0)=V(T)=\gamma$, as in Figure~\ref{fig::examplechart}.  Then in any optimal solution, $\beta(t)$ alternates between $\bmin$ and $\bmax$ finitely many times, and the $V(t)$ graph (like the one in Figure~\ref{fig::examplechart}) has finitely many excursions away from $\gamma$, which alternatively go above $\gamma$ ($\bmax$ on the way up, $\bmin$ on the way down) or below $\gamma$ ($\bmin$ on the way down, $\bmax$ on the way up).  All of these excursions have maximal area (i.e., they correspond to $\log I(t)$ crossing from $C_1$ to $C_2$ or back) except possibly for two smaller-but-equal-area excursions (which together correspond to $\log I(T)$ crossing from one of $\{C_1, C_2 \}$ to an intermediate value and back).  For more general boundary conditions, if an optimal solution $V$ hits $\{C_1,C_2\}$ at least once, and $t_1$ and $t_2$ are the first and last times this happens, then the restriction of $V$ to $[t_1,t_2]$ behaves as described above, while the restrictions to $[0,t_1]$ or $[t_2,T]$ each have the form described in Proposition~\ref{prop::bestpossible}.  If $V$ never hits a wall, then it must be of the form described in Proposition~\ref{prop::bestpossible}.
\end{prop}

Proposition~\ref{prop::wallconstrained} formalizes the notion, suggested by Figures~\ref{fig::trucks} and~\ref{fig::examplechart}, that when $T$ is large, the optimal long-term strategy is to alternate between maximal forward acceleration and maximal reverse acceleration, timing the accelerations so that the car's velocity reaches $\gamma$ exactly as it reaches each truck.  On the other hand, one may have to break the Figure~\ref{fig::examplechart} pattern at a couple of turn-around points to ensure that the boundary conditions are satisfied.  The proof is similar to the proof of Proposition~\ref{prop::bestpossible}. One checks that if $V(t)$ does not have the asserted form then it is possible to make modifications to increase $\int V(t)^2dt$ while keeping $\int V(t)dt$ the same {\em and} keeping $\log I(t)$ within bounds. See Appendix~\ref{app::proofs}.

\subsection{Multiple subpopulations} \label{sec::subpopulations}
Suppose there are several {\em disjoint} subpopulations (factory workers in Town A, factory workers in Town B, students/teachers in Town A, students/teachers in Town B,  etc.) each of which has some amount of flexible activity that can be set independently. Suppose further that there is some interaction between members of different groups that does not depend on the level of flexible activity (e.g., because a student and a factory worker live in the same household). To formalize this, for $1 \leq k \leq n$ write \begin{equation} \label{eqn::note2} \dot E_k(t) = -E_k(t) + \beta_k(t) I_k(t) + \sum_{j \neq k} \alpha_{j,k} I_j(t), \,\,\,\,\,\,\,\,  \,\,\,\,\, \dot I_k(t) = E_k(t) - \gamma I_k(t),\end{equation} where for each $k$ the process $\beta_k:[0,T] \to [\bmin,\bmax]$ is a control parameter (a measurable function of $t$).  Letting $V_k(t) := E_k(t)/I_k(t)$, we find

\begin{align} \label{eqn::Yk} \dot V_k(t) &= \frac{ \dot E_k(t) I_k(t) - \dot I_k(t)E_k(t)}{I_k(t)^2} \\
&=  \frac{ \bigl(- E_k(t) + \beta_k(t)I_k(t) +  \sum_{j \neq k} \alpha_{j,k} I_j(t)\bigr) I_k(t) - \bigl(E_k(t) - \gamma I_k(t) \bigr) E_k(t)}{I_k(t)^2}  \nonumber \\
&=  \frac{  -E_k(t)I_k(t) + \beta_k(t) I_k(t)^2 - E_k(t)^2 + \gamma I_k(t)E_k(t) +  \sum_{j \neq k} \alpha_{j,k} I_j(t)\bigr) I_k(t)}{I_k(t)^2}    \nonumber \\
&= \beta_k(t) -V_k(t)^2 + (\gamma-1)V_k(t) + \sum_{j \neq k} \alpha_{j,k} I_j(t)/I_k(t)  \nonumber \\
&= \beta_k(t) -\phi(V_k(t)) +  \sum_{j \neq k} \alpha_{j,k} e^{X_j(t) - X_k(t)}, \nonumber
\end{align}
where again $\phi$ is the quadratic function defined by $\phi(x) := x^2 + (1-\gamma)x$. This also implies \begin{equation}\label{eqn::betak} \beta_k(t) = \dot V_k(t) + \phi(V_k(t))-  \sum_{j \neq k} \alpha_{j,k} e^{X_j(t) - X_k(t)}.\end{equation} We then have
\begin{align} \label{eqn::Ak}
A &= \sum_{k=1}^n \int_0^T \Bigl( \phi\bigl(V_k(t)\bigr) - \sum_{j \neq k} \alpha_{j,k} e^{X_j(t) - X_k(t)}\Bigr) dt + V_k(T)-V_k(0) \\ &= \sum_{k=1}^n\Bigl( \int_0^T  V_k(t)^2 dt  +  \int_0^T (1-\gamma) V_k(t) dt + V_k(T) - V_k(0) - \int_0^T \sum_{j \neq k} \alpha_{j,k} e^{X_j(t) - X_k(t)} dt\Bigr) . \nonumber
\end{align}

Removing the terms that depend only on the given boundary values, the objective becomes

\begin{equation} \label{eqn::Ak2}
\sum_{k=1}^n\Bigl( \int_0^T  V_k(t)^2 dt\Bigr)  - \sum_{j \neq k}\Bigl( \int_0^T \alpha_{j,k} e^{X_j(t) - X_k(t)} dt \Bigr),
\end{equation}
subject to the bounds on $\beta_k(t)$ and whatever initial and final values for the $V_k$ and $X_k$ are assumed.  If we assume further that $\alpha_{j,k} = \alpha_{k,j}$ (which might make sense if the subpopulations are of similar size) then we find that $A$ is a constant plus

\begin{equation} \label{eqn::Ak3}
\sum_{k=1}^n\Bigl( \int_0^T  V_k(t)^2 dt\Bigr)  - \sum_{j \neq k}\Bigl( \int_0^T \alpha_{j,k} \cosh(X_j(t) - X_k(t)) dt \Bigr).
\end{equation}

The first term is a measure of policy oscillation: it is large if the velocities $V_k(t)$ have large swings. The latter term is a measure of policy coordination: it is largest if prevalence does not differ too much from one subpopulation to another.  The fact that $A$ is equal to (a constant plus) \eqref{eqn::Ak3} can be summarized in English as follows:  {\bf once boundary values are fixed, activity is largest when policies are {\em coordinated} but {\em oscillatory}.}  On the other side, one might expect that when $T$ is large, activity-{\em minimizing} policies could involve traveling for long stretches in ``migrating bird'' patterns like the one in Figure~\ref{fig::vcars}, where only one subpopulation (or a small number of them) is substantially active, and others acquire small amounts of infection at a steady rate from the active subpopulations, despite being themselves relatively inactive.

\begin{figure}[ht!]
\begin{center}
\includegraphics[scale=0.7]{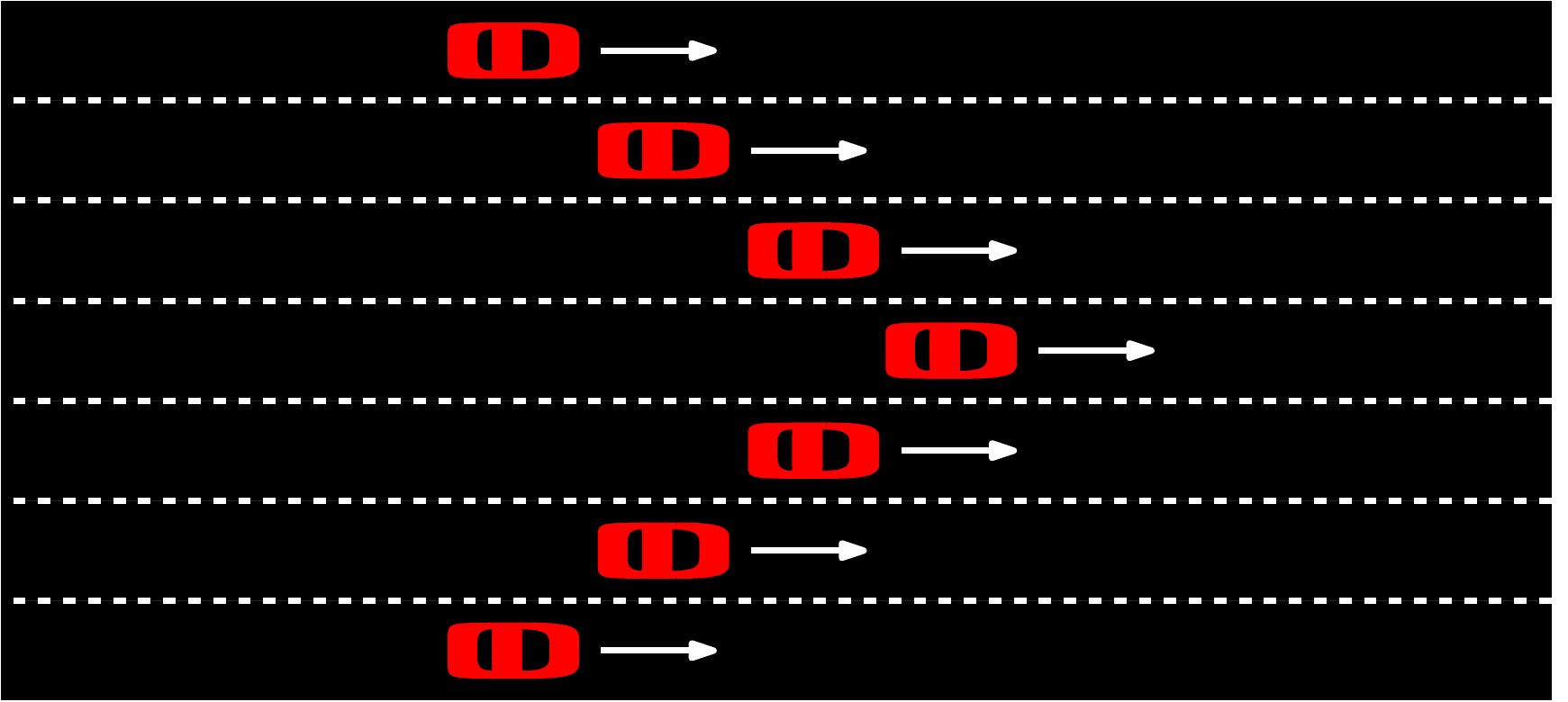}
\caption{\label{fig::vcars} Adapting the analogy in Figure~\ref{fig::trucks}, let $X_k(t)$ and $V_k(t)$ denote the rightward position and velocity of the $k$th car from the top. Let $\beta_k(t)$ be the corresponding internally generated force, describing how much fuel the $k$th car is using.  Assume that $\alpha_{j,k}=1$ if $|j-k|=1$ and $0$ otherwise (so cars only influence their neighbors). Per \eqref{eqn::Ak2}, if $|j-k| = 1$, then the amount of forward drafting force the $j$th car induces on the $k$th car is $e^{X_j(t)-X_k(t)}$. This force is small (but positive) if the $j$th car is behind the $k$th car, but it becomes large as the $j$th car gets far ahead of the $k$th car (which prevents the cars from getting too far apart). Assume that all cars have velocity $\gamma$ but the middle car uses a lot of fuel (so $\beta_4>> \bmin$) while the others use very little thanks to drafting forces (so that $\beta_k \approx \bmin$ for $k \not = 4$).  If a pattern like this is the {\em best possible} for saving fuel, then it is the {\em worst possible} for boosting activity.  In Figure~\ref{fig::examplechart} the activity gap between best and worst policy was 20 percent. In this figure, 6 of the 7 cars have very low $\beta(t)$ and (depending on the parameters) keeping all $X_k$ equal might allow for {\em several times} as much activity.
}
\end{center}
\end{figure}

\section{Discussion}
There are many factors we have not considered: implementation costs, inoculum size, contract tracing, herd immunity effects (which may be significant in subpopulations even if overall prevalence is low), unpredictable super-spreader events at low prevalence (perhaps quickly boosting cases from 1 per million to 100 per million), regional disease-elimination opportunities, subpopulation differences, etc. But at least within the models presented here, coordinated suppression (followed by relaxation and resuppression as needed) appears superior to temporal consistency and subpopulation variability. As noted in \cite{lockdownscount2020}, the benefits of staggering flexible activity are smaller when there is less flexible activity to stagger---but larger when more realistic incubation/infectious period distributions are incorporated into the model.

These findings recall the clich\'e that if everyone perfectly distanced for three weeks the disease would disappear. The clich\'e does not take into account that some contact cannot be eliminated, some infections last unusually long, etc. But this paper shows that within simple models that do account for these things, the same principle applies: people enjoy more contact overall when their activity is coordinated.

It may be {\em especially ineffecient} for some subpopulations to tightly close for the long term while others remain open enough to maintain a steady disease prevalence, as in the V-shaped pattern from Figure~\ref{fig::vcars}. We remark that one can imagine this type of pattern arising with no government action at all --- e.g., if individuals voluntarily reduce activity once prevalence nears a threshold, but that threshold differs among subpopulations. It could also arise if subpopulations pull in opposite directions due to differing preferences or needs; perhaps $\bmin$ and $\bmax$ differ from group to group, or perhaps some prefer, all things considered, to acquire substantial herd immunity through infection, while others prefer to keep prevalence low. There are many social, political and game theoretic issues we won't discuss.

Instead, we conclude by reiterating our main point: within the simple SEIR-based models discussed here, the {\em amount} of activity a society enjoys is higher when the activity is {\em staggered} and {\em coordinated}.

\begin{appendix}
\section{Proofs} \label{app::proofs}

\begin{proof}[Proof of Proposition~\ref{prop::bestpossible}]
It is not hard to see that the differentiable functions $V:[0,T] \to (0,\infty)$ that satisfy \eqref{eqn::Yconstraint} are precisely those that (in the language of Figure~\ref{fig::ychart}) never cross a blue curve from above to below and never cross a red curve from below to above.  More generally (without assuming differentiability of $V$) one may take this as a formal definition of what it {\em means} to satisfy \eqref{eqn::Yconstraint}.  Such functions are Lipschitz and hence differentiable outside of a set of Lebesgue measure zero by Rademacher's theorem, but $V$ may have points of non-differentiability, since $\beta$ may have discontinuities.

Taking this view, it is clear that the set of $V$ that satisfy these constraints---and have the given values of $V(0)$, $V(T)$, and $\int_0^T V(t)dt$ --- is compact w.r.t.\ the $L^\infty$ norm (precompactness follows from Arzel\`a-Ascoli and the constraint is clearly preserved under $L^\infty$ limits).

Suppose the given values for $V(0)$, $V(T)$, and $\int_0^T V(t)dt$ are such that there exists at least one $V$ with these values that satisfies \eqref{eqn::Yconstraint}. (The proposition statement holds trivially otherwise.) Then the existence of an {\em optimal} $V(t)$ follows from the continuity of $\int_0^T V(t)^2dt$ w.r.t.\ the $L^\infty$ norm and the above-mentioned compactness.  We aim to show that any such $V$ has the form described in the proposition statement.

Given any $V$ satisfying \eqref{eqn::Yconstraint}, we define a point $t \in (0,T)$ to be {\em taut} if either $\beta(t) = \bmax$ a.e.\ in a neighborhood of $t$ or $\beta(t) = \bmin$ a.e.\ in a neighborhood of $t$. In other words, $t$ is taut if $V(t)$ traces a blue curve or a red curve in a neighborhood of $t$.  We say that $t$ is a {\em sharp peak} if $\beta(t) = \bmax$ a.e.\ in $(t_1,t)$ and $\beta(t) = \bmin$ a.e.\ in $(t,t_2)$ for some $t_1<t<t_2$.  In other words, $V(t)$ traces a red curve to the left of $t$ and a blue curve to the right. Define a {\em sharp valley} analogously.  Call $t$ {\em upward-flexible} if it is neither taut nor a sharp peak. If $t$ is upward-flexible then one can modify $V(t)$ (shifting it ``upward'') in any small neighborhood of $t$ in a way that increases $\int_0^T V(t)dt$ by any sufficiently small amount while respecting  \eqref{eqn::Yconstraint}.  This can be done for example by replacing $V$ with the supremum of $V$ and a function that is taut except for a sharp peak that lies just above $V$ near $t$. The analogous statement holds if $t$ is {\em downward-flexible}, i.e., neither taut nor a sharp valley. Call a point {\em doubly flexible} if it is both upward and downward flexible.

Note that if $V(s) > V(t)$, and $s$ is upward-flexible and $t$ is downward-flexible, then one can increase $V$ in a neighborhood of $s$ and compensate by decreasing $V$ in a neighborhood of $t$ in a way that increases $\int_0^T V(t)^2dt$ while keeping $\int_0^T V(t)dt$ the same.  We conclude from this that if $V$ is optimal, the supremum of $V(s)$ over upward-flexible $s$ is at most the infimum of $V(s)$ over downward flexible $s$. In other words, there exists a $v$ such that all points on the graph of $V$ below height $v$ are either taut or sharp valleys, and all points on the graph of $V$ above height $v$ are either taut or sharp peaks.  Similar arguments show that $V(t)$ cannot be locally constant at $v$ unless $v \in \{\vmin, \vmax \}$ (in which case the height $v$ is crossed only once). 

If $V$ is optimal and {\em monotone non-increasing} (i.e., in a region where the red and blue curves are downward---this can happen in Figure~\ref{fig::ychart} if $V(0)$ and $V(T)$ are both greater than $\vmax$) then the above implies that there can be at most one sharp peak above $v$ (where $\beta(t)$ changes from $\bmax$ to $\bmin$) and at most one sharp valley below $v$ (where $\beta(t)$ changes from $\bmin$ to $\bmax$) and no doubly flexible points, which implies the proposition statement in this special case (similarly if $V$ is monotone non-decreasing).

If $V$ is not monotone, it still follows from the above that any excursions away from the horizontal line of height $v$ are either $V$-shaped (blue curve going down, red curve going up) or $\Lambda$-shaped (red curve going up, blue curve coming down). 
Define an {\em occupation measure} $\nu$ by letting $\nu(S)$ denote the Lebesgue measure of $\{t \in [a,b]: V(t) \in S \}$.  Clearly the occupation measure determines $\int_0^T V(t)dt$ and $\int_0^T V(t)^2dt$. So if $V$ is optimal then any $W$ with the same occupation measure, which also satisfies \eqref{eqn::Yconstraint}, must be optimal as well.

Now we will argue that if $V$ crosses any given horizontal line (i.e., passes from above to below or vice versa) more than twice then $V$ is not optimal. The idea is explained in Figure~\ref{fig::movingpeak}. Suppose (for sake of getting a contradiction) that there exists an interval $(t_1,t_2) \subset (0,T)$ such that $V(t_1) = V(t_2)$ and $V(t)>V(t_1)$ for $t \in (t_1,t_2)$. Suppose that there also exists another interval $(t_3,t_4) \subset (0,T)$, of positive distance from $(t_1,t_2)$, such that either $V(t_3)=V(t_1)$ or $V(t_4)=V(t_1)$ and $V(t)>V(t_1)$ for $t \in (t_3,t_4)$.  Then, as explained in Figure~\ref{fig::movingpeak}, one can ``rearrange'' $V$ to create a $W$ with the same occupation measure but with doubly flexible points of different heights.  The first steps in the figure should be self-explanatory, but the final ``flattening then $N$'' step requires explanation. Suppose generally that there is an interval $[a,b]$ such on this interval $V$ achieves its minimum at $a$ and its maximum at $b$ but $V$ is not monotone on $[a,b]$. We can then construct the unique continuous function $W(t)$ that agrees with $V$ outside of $(a,b)$, that has the same occupation measure as $V$, and that is monotone non-decreasing.  In other words, $W$ ``spends the same amount of time at each vertical height'' as $V$ but $W$ is non-decreasing (so it visits the heights strictly in increasing order). 
  It is then clear that $W$ satisfies \eqref{eqn::Yconstraint} and also that the slope is non-extremal at points of different heights, so that $W$ (and hence $V$) is suboptimal, which is the desired contradiction.

\begin{figure}[ht!]
\begin{center}
\includegraphics[scale=0.8]{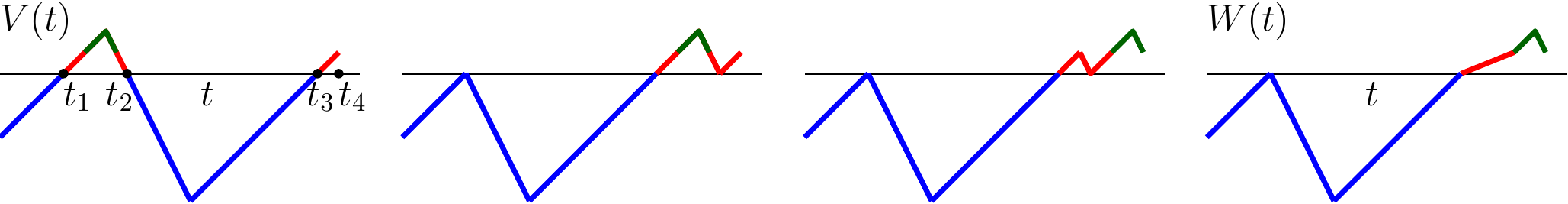}
\caption{\label{fig::movingpeak} First figure transformed into second figure by ``swapping'' red/green peak and blue valley. Second transformed into third by swapping green peak and red valley. Third transformed into fourth by ``flattening'' the red $N$ in a way that preserves the occupation measure but produces a curve of non-extremal slope.
}
\end{center}
\end{figure}

Now suppose $V(T) > V(0)$.  Since $V$ crosses every horizontal line at most twice, $V$ must be monotone non-decreasing between $V(0)$ and the first time $t_1$ at which $V$ reaches its maximum, monotone non-increasing between $t_1$ and the first time $t_2$ at which $V$ reaches its minimum, and monotone non-increasing between $t_2$ and $T$.  (We allow for the degenerate possibility that $0=t_1$ or $t_2 = T$.) Since $V$ has at most one height at which doubly flexible points occur, the above conditions (and the fact that $V$ cannot be locally constant) imply that $V$ cannot have any doubly flexible points, and indeed must have the form required in the proposition statement. A similar argument applies if $V(T) < V(0)$. If $V(T) = V(0)$, the same argument shows that $V$ must have one upward and one downward excursion away from its initial position---but these excursions can be made in either order, so there are two equally optimal solutions in this case. (The one that puts the downward excursion first would yield fewer {\em infections} but these are not included in the definition of $A$.) \end{proof}

\begin{proof}[Proof of Proposition~\ref{prop::wallconstrained}]

The arguments in the proof of Proposition~\ref{prop::bestpossible} imply that each excursion of $V$ above $\gamma$ must obtain the maximal possible slope (with $\beta(t)=\bmax$) on an initial interval and the minimal possible slope (with $\beta(t)=\bmin)$ for the rest of the excursion; in other words, in the language of Figure~\ref{fig::ychart}, its graph is a concatenation of a red curve and a blue curve. If this were not the case, we could apply the procedures in the proof of Proposition~\ref{prop::bestpossible} to this interval and produce a $W$ with higher second moment.  (Note that if we apply these procedures to an interval on which $V(t)>\gamma$, they do not change the fact that $\log I(t)$ is increasing over the course of that interval and they do not change the amount of increase, so they cannot lead to a violation of the $C_1 \leq \log I(t) \leq C_2$ constraint.)  Similarly, each excursion of $V$ below $\gamma$ consists of a maximally downward segment (blue) followed by a maximally upward segment (red).

Let us take the same argument a bit further. Suppose $V(t)$ is optimal and suppose that that $\log I(t)$ does not hit $C_1$ or $C_2$ during an interval $[s_1, s_2]$.  Then we claim that every local maximum (minimum) of $V$ in $(s_1,s_2)$ must be a global maximum (minimum).  First observe that $V$ must have derivative in $\{\bmin,\bmax\}$ almost everywhere within $(s_1,s_2)$, since otherwise (sufficiently small) perturbations like those described in the proof of  Proposition~\ref{prop::bestpossible} would increase  $\int_0^T V(t)^2 dt$ without violating the $C_1$ and $C_2$ conditions. Next, suppose that a local (but not global) maximum is obtained at $s$.  Then (since $V$ is not locally constant) by choosing arbitrarily small $\epsilon$, we can arrange so that the component of $\{t: V(t) > V(s) - \epsilon \}$ containing $s$ is arbitrarily small but non-empty. We can then ``redistribute'' the local time corresponding to that component elsewhere, as in Figure~\ref{fig::movingpeak}, to produce a $W$ with the same occupation measure as $V$, and if $\epsilon$ is small enough this redistribution will not change the fact that $C_1$ and $C_2$ fail to be hit, but it will also produce a positive mass of places where $\beta(t) \not \in \{\bmin, \bmax \}$, which enables an improvement to  $\int_0^T V(t)^2 dt$, which is a contradiction.  We conclude from this that within $[s_1,s_2]$, the function $V$ must take the form described in Proposition~\ref{prop::bestpossible}.  By taking limits, we find that the same is true if either or both of $\log I(s_1)$ and $\log I(s_2)$ lie in $\{C_1, C_2 \}$ but $V(s) \not \in \{C_1, C_2 \}$ for $s \in (s_1, s_2)$.

If $\log I(s_1) = \log I(s_2) = C_1$ then we must have $V(s_1) = V(s_2) = \gamma$ and in between $s_1$ and $s_2$ (recalling the statement of Proposition~\ref{prop::bestpossible}) $V$ makes one upward and one downward excursion away from $\gamma$, and the areas between these curves and the horizontal line at height $\gamma$ both have equal area as in Figure~\ref{fig::examplechart}. This area must be strictly less than $C_2-C_1$ if $C_2$ is never hit in $(s_1,s_2)$.  In this case (and the analogous case with the roles of $C_1$ and $C_2$ reversed) we refer to $(s_1, s_2)$ as a ``single wall excursion'' (since the same element of $\{C_1,C_2 \}$ is hit at both endpoints).  Similarly, if $\log I(s_1) = C_1$ and $\log(s_2) = C_2$ (but these two endpoints are avoided for $(s_1,s_2)$) then between $s_1$ and $s_2$ the function $V$ must make a single positive excursion enclosing the maximal possible area above $\gamma$, namely area $C_2 - C_1$, precisely as in Figure~\ref{fig::examplechart}. In this case we refer to $(s_1,s_2)$ as a ``double wall excursion.''

Next, we will argue that if there are {\em two} single wall excursions, then one can make one of the excursions bigger and the other one smaller in a way that does not change $\int_0^t V(t)dt$ and $\int_0^t V(t)^2dt$ but that makes it so that $\int_0^t V(t)^2dt$ is no longer maximal (a contradiction).  One of the two single wall excursions (call it ``smaller'') must have size less than or equal to that of the other (larger) one.  As we have done before (in  Figure~\ref{fig::movingpeak}) we can then ``move mass'' from near the tips of the corresponding upper/lower excursions of $V$ (away from $\gamma$) in the smaller one to the upper/lower excursions of $V$ (away from $\gamma$) in the larger one in a way that produces a new function that is suboptimal, and this yields the contradiction.

We conclude that there is at most one single wall excursion, and the rest of the proposition follows.
\end{proof}

\section{More general utility functions} \label{app:generalutility}
\subsection{Setup and motivation}

We now generalize our original setup to account for crowding effects.  In this setting, we consider two kinds of activity: first, $\mu(t)$ is the amount of {\em useful activity} taking place at time $t$.  Informally, think of $\mu(t)$ as encoding the (risk-weighted) number of conversations, restaurant meals, haircuts, etc. Second, $\beta(t)$ is the amount of {\em transmission activity} that drives the ODE \eqref{eqn::note}. Instead of making these quantities equal as before, we now assume they are related by $\mu(t) = u(\beta(t))$ where $u$ is increasing and continuous but potentially {\em non-linear}. We then denote the {\em total useful activity} by $U =\int_0^T \mu(t)dt =  \int_0^T u(\beta(t)) dt$ and assume that our goal is to maximize $U$, instead of maximizing $A = \int_0^T \beta(t)dt$.\footnote{In the car analogy, $\mu(t)$ is fuel use, $\beta(t)$ is force, but the relationship is non-linear. Maximizing fuel use is again the goal.}

One way to motivate this is to suppose that some fraction of the disease transmission $\beta(t)$ comes from {\em deliberate close interaction} (i.e., is spread between friends or coworkers choosing to engage in a valuable activity together) and that the rest comes from infectious air lingering in public places (hallways, subway cars, etc.) The former might be linear in $\mu(t)$ (twice as many conversations means twice as many chances for spread) but the latter might be quadratic in $\mu(t)$ (if the density of infectious particles in the air and the number of people inhaling them are both linear in $\mu(t)$).  Combining the two effects, we might find $\beta(t) = \psi( \mu(t)):= a_1 \mu(t) + a_2 \mu(t)^2$ where $a_1$ and $a_2$ are positive constants, and taking the positive inverse, \begin{equation}\label{eqn::ub} u(x) = \psi^{-1}(x)=
\frac{-a_1 + \sqrt{a_1^2 + 4a_2x}}{2a_2}.\end{equation} As a concrete example, suppose $a_1 = 3/4$ and $a_2=1/4$.  If $\beta(t)=\mu(t)=1$ then $25$ percent of the transmission comes from ``lingering air'' (the quadratic term). If $\beta(t)=2.25$ then (after solving for $\mu(t)$) about 38 percent of the transmission comes from lingering air; if $\beta(t)=.25$ it is only about $9$ percent. In this example, ``crowding effects'' play a larger role when activity is higher.\footnote{The assumption that $\mu(t)$ and $\beta(t)$ determine one another via a single function $u$ is a simplification. In principle, the relationship between $\mu(t)$ and $\beta(t)$ could vary in time. For example, perhaps deliberate interaction is a larger factor for weekend activities and lingering air is a larger factor on weekdays, so that the same $u$ cannot be used for both.}

We stress that $\mu(t)$ is a measure of the {\em amount} of useful activity---defined in a way that ensures $\beta(t) = u^{-1}(\mu(t))$.  It is {\em not} a measure of the {\em value} of the activity. If a more-valued activity causes the same amount of disease transmission as a less-valued activity, then it will make the same contribution to $\mu(t)$. We allow for the possiblity that in scenarios where $U$ is kept small, only very important activity will be allowed, but in scenarios where $U$ is large, more discretionary activity will occur. We only assume that utility is an {\em increasing function} of $U$ so that maximizing $U$ is a reasonable objective.

We can generalize Problem~\ref{prob::activity} by replacing $A$ with \begin{equation} \label{eqn::Udef} U = \int_0^T u\Bigl( \beta(t)\Bigr)dt = \int_0^T u \Bigl( \dot V(t) + \phi(V(t)) \Bigr)dt,\end{equation} where $u$ is some fixed twice-differentiable function, and as before we write $\phi(x) = x^2 + (1-\gamma)x$. The setup in Problem~\ref{prob::activity} amounts to taking $u(b) = b$ for $b \in [\bmin,\bmax]$.

Now let us express \eqref{eqn::Udef} a different way. Write
\begin{equation}\label{eqn::Gy} G_y(x) := u(x+\phi(y)) - u(\phi(y)) - u'(\phi(y))x. \end{equation}  Observe that for all $y$ we have $G_y(0) = G_y'(0)=0$.
If $u$ is (strictly) concave then (for fixed $y$) $G_y$ is (strictly) concave, and has a maximum at $0$. To ensure that \eqref{eqn::Gy} makes sense for relevant inputs, it will be convenient for us to extend the definition of $u$ beyond $[\bmin,\bmax]$ to the full range of $\phi$ (which is all of $[0,\infty)$ when $\gamma=1$) in such a way that $u$ remains concave. It does not matter exactly how we do this, but one natural approach is to assume $u$ is differentiable everywhere but affine---or strictly concave but nearly affine---outside of $[\bmin,\bmax]$. Now we can rewrite \eqref{eqn::Gy} with $x = \dot V(t)$ and $y = V(t)$ to obtain

\begin{equation}
u \Bigl( \dot V(t) \!+ \! \phi(V(t)) \Bigr) \!=\! G_{V(t)}\bigl(\dot V(t)\bigr) + u\bigl(\phi(V(t))\bigr)  + u'\bigl(\phi(V(t))\bigr)\dot V(t)
\end{equation}
and substituting this into \eqref{eqn::Udef} yields
\begin{equation} U= \int_0^T u \Bigl( \dot V(t) \!+ \! \phi(V(t)) \Bigr)dt \!=\! \int_0^T \Bigl( G_{V(t)}\bigl(\dot V(t)\bigr) + u\bigl(\phi(V(t))\bigr)  + u'\bigl(\phi(V(t))\bigr)\dot V(t) \Bigr)dt .\end{equation}

The final RHS term can be written $u'\bigl(\phi(V(t))\bigr)\dot V(t) = \frac{\partial}{\partial t} r(V(t)) = r'(V(t)) \dot V(t)$ where $r'(x) = u'(\phi(x))$, i.e., $r(a) := \int_0^a u'(\phi(x))dx.$  Thus the final RHS term integrates to a quantity that depends only on $V(T)$ and $V(0)$.
If we remove that, the objective becomes \begin{equation}\label{eqn::phi}\int_0^T  u\bigl(\phi(V(t))\bigr) dt + \int_0^T  G_{V(t)}\bigl(\dot V(t)\bigr) dt\end{equation}
Writing it this way, we have separated the objective into two pieces: the first term ascribes different benefits to different velocities via the function $u \circ \phi$. The second (non-positive) term ascribes costs to non-zero acceleration rates (in a manner that also depends on velocity).

We can now formulate the problem in the language of Problem~\ref{prob::activity2} as follows:

\begin{prob} \label{prob::activityu}
Given $V(0)$, $V(T)$, and $\int_0^T V(t) dt$, find a $V$ that maximizes \eqref{eqn::phi} subject to \eqref{eqn::Yconstraint}.
\end{prob}

Now suppose that $u(x)$ is concave in $x$ but $u(\phi(x))$ is strictly convex.  The latter holds if $u(x) = x$ but fails if $u$ is ``too concave.'' To illustrate what this means, consider the $\gamma=1$ case.\footnote{In a probabilistic formulation of the SEIR model, setting $\gamma=1$ corresponds to assuming that the incubation time and the infectious time are independent exponential random variables with the same rate $1$.  If $f(x) = x e^{-x}$ is the density function for the {\em sum} of these positive random variables then $\int_a^b f(t) \beta(t) dt$ is the expected number of people infected between time $a$ and $b$ by a person infected at time $0$.  If $\gamma$ is either very small or very large, then the corresponding $f$ is approximately exponential, and the model is effectively more like an SIR model; taking $\gamma$ close to one ensures that $f$ is more concentrated (a smaller standard deviation relative to its mean). If the true $f$ is {\em actually much more} concentrated than $xe^{-x}$ then the models of this paper are inadequate, and a different approach is needed (such as Erlang SEIR with a higher Erlang parameter, see \cite{lockdownscount2020}). Still, $\gamma=1$ might be the best approximation {\em within} the framework of this paper.} In this case $\phi(x) = x^2$ so $u(\phi(x)) = u(x^2)$ is strictly convex if $u(x) = x^\alpha$ for $\alpha \in (1/2,1]$, but not if $u(x) = x^\alpha$ for $\alpha \leq 1/2$. Note that $u(x^2)$ is also strictly convex if $u$ is as given in \eqref{eqn::ub}, for any positive $a_1$ and $a_2$.

Since the mean of $V$ is fixed, the first term of \eqref{eqn::phi} is the worst possible if $V$ is constant (by Jensen's inequality). But the second term penalizes fluctuation (i.e., one pays a price for non-zero derivative) so these two factors work against each other.  On the other hand, if $V$ varies slowly, the second term should not matter very much. Proposition~\ref{prop::worstwithu} below is a simple illustration of that point.

\subsection{Best and worst policies}
If $u$ is concave, then a rapidly fluctuating $\beta$ may yield a lower $U$ than a constant $\beta$ with the same mean so that (in contrast to Proposition~\ref{prop::worstpossible}) constant strategies are not the worst possible. On the other hand, Proposition~\ref{prop::worstwithu} states that constant $V(t)$ are the worst possible (given the corresponding boundary data) among functions that {\em vary slowly} in the sense of having no Fourier modes of short wavelength.

\begin{prop}\label{prop::worstwithu}
Suppose that $u(x)$ is smooth and concave in $x$ but $u(\phi(x))$ is smooth and strictly convex. 
Then there exists a $C>0$ (independent of $T$ or the boundary data) such that constant-$V(t)$ strategies are the worst possible (i.e, $U$-minimizing, given constraints from Problem~\ref{prob::activityu}) among all differentiable $V(t)$ whose Fourier series decompositions on the interval $[0,T]$ include no mode with wave length less than $C$.
\end{prop}

Note that this proposition holds trivially if $T < C$, and hence provides no information in that case. It does not rule out the possibility that constant strategies are optimal over very short time periods.

\begin{proof}
Write $v$ for the fixed value of $V(0)=V(T)= T^{-1}\int_0^T V(t)dt$.  Let $L$ be the affine function tangent to $u\circ \phi$ at $v$, and write $\tilde u = u \circ \phi - L$.  Then we can write the  first term of \eqref{eqn::phi} as $\int_0^T L(V(t))dt + \int_0^T \tilde u (V(t))dt$. Since $\int_0^T L(V(t))dt$ is fixed by the boundary data, we can ignore that term, so the objective becomes 

\begin{equation}\label{eqn::phi2}\int_0^T  \tilde u \bigl(V(t))\bigr) dt + \int_0^T  G_{V(t)}\bigl(\dot V(t)\bigr) dt\end{equation}

We can assume that $V(0) = V(T) \in (\vmin,\vmax)$ (since otherwise there would only be one or zero possible solutions with the same boundary values for $V$ and with $\int_0^T V(t) = T V(0)$, and the proposition statement would be trivially true). Thus, in $\eqref{eqn::phi2}$ we need only to evaluate $\tilde u(v)$ for $v \in (\vmin,\vmax)$ and $G_{v}(x)$ for $v \in (\vmin,\vmax)$ and $x$ in the {\em bounded range} of $\dot V(t)$ values possible when $V(t) \in [\vmin,\vmax]$.

Within this range, because of the convexity and smoothness assumptions, there exists a $c_1>0$ such that $\tilde u (x) \geq c_1 (x-v)^2$ and there exists a $c_2$ such that $0 \geq G_v(x) \geq - c_2 x^2$. Thus 

\begin{equation}\label{eqn::phi3}\int_0^T  \tilde u \bigl(V(t))\bigr) dt + \int_0^T  G_{V(t)}\bigl(\dot V(t)\bigr) dt \geq c_1 \int_0^T (V(t)-v)^2 dt - c_2 \int_0^t \dot V(t)^2 dt\end{equation}

If we write $V_k(t) = e^{k 2\pi i t /T}$ for the $k$th Fourier mode, then $\int_0^T |\dot V_k(t)|^2dt = (2\pi k/T)^2 \int_0^T |V_k(t)|^2dt$.  As long as $(2\pi k/T)^2 \leq c_1/c_2$, \eqref{eqn::phi2} will be negative if we set $V(t) = v+V_k$, and by orthogonality of the Fourier series, the same applies to any linear combination of Fourier modes $V_k$ such that $(2\pi k/T)^2 \leq c_1/c_2$, or equivalently $k/T \leq \sqrt{c_1/c_2}/(2\pi)$,  which means that the wavelength $T/k$ satisfies $T/k \geq 2\pi \sqrt{c_2/c_1}$. \end{proof}

\begin{prop} \label{prop::bestwithu}
In the context of Proposition~\ref{prop::worstwithu}, the best possible ($U$-maximizing) $V(t)$ cannot cross any horizontal line more than twice in $(0,T)$. If $V(0) < V(T)$ and we write $m:=\inf_{t \in [0,T]}V(t)$ and $M:=\sup_{t \in [0,T]}V(t)$, then $V$ must be monotone non-increasing between time $0$ and the first time it hits $m$, then monotone non-decreasing until the first time it hits $M$, then monotone non-increasing again until time $T$.  (Similar statements hold if $V(0) > V(T)$ or $V(0)=V(T)$, c.f.\ Proposition~\ref{prop::bestpossible}.)
\end{prop}

\begin{proof}
The argument in Figure~\ref{fig::movingpeak} works exactly the same way in this setting; the only difference is that for the fourth ``flattening'' step, one can use Jensen's inequality to show that utility is {\em strictly larger} for the flattened curve than for the original.  The flattening does not change the first term \eqref{eqn::phi2}, but it makes the second term strictly larger. That is, we claim that if $W$ is a curve produced by flattening $V$ on $[s_1,s_2]$ then \begin{equation}\label{eqn::gvt} \int_0^T G_{V(t)} (\dot V(t))dt < \int_{s_1}^{s_2} G_{W(t)} (\dot W(t))dt.\end{equation}
To see this, note that the fundamental theorem of calculus and the construction of $W$ imply \begin{equation}\label{eqn::vtwt} \int_{t\in[s_1,s_2]: V(t) \in (a,b)} \dot V(t) dt  = \int_{t\in[s_1,s_2]: W(t) \in (a,b)} \dot W(t)dt.\end{equation}
If $t$ is sampled uniformly from $[s_1,s_2]$ then there is an $F$ such that $\mathbb E[\dot V(t) | V(t)] = F(V(t))$ and since \eqref{eqn::vtwt} holds for any $(a,b)$ this implies $E[\dot W(t) | W(t)] = F(W(t))$.  On the other hand, since $\dot V(t)$ assumes negative values with positive probability and $\dot W(t)$ is conditionally deterministic given $W(t)$, we have a strict inequality on conditional variance: $$\mathbb E\bigl[\Var[\dot V(t) | V(t)]\bigr] > \mathbb E \bigl[\Var[\dot W(t) | W(t)]\bigr].$$  Since there is (on the range inputs possible here) a negative upper bound on the second derivative of $G_{V(t)}$ we deduce that
 $$\mathbb E \Bigl[ \mathbb E \bigl[G_{V(t)} (\dot V(t)) | V(t)\bigr] \Bigr]  < \mathbb E \Bigl[ \mathbb E\bigl[G_{W(t)} (W(t)) | W(t)\bigr] \Bigr],$$
which implies \eqref{eqn::gvt}.

This argument shows that $V$ cannot cross any horizontal line more than twice. The rest of the proposition statement easily follows from this.
\end{proof}

\begin{figure}[ht!]
\begin{center}
\includegraphics[scale=0.47
]{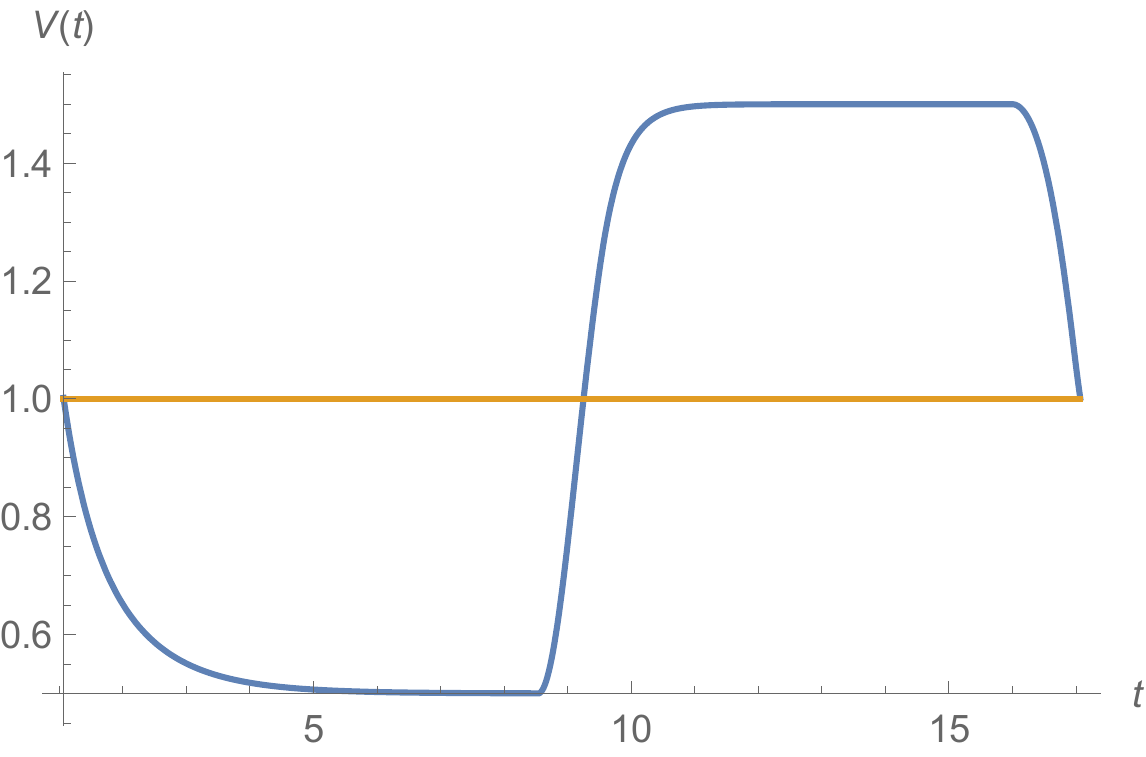}\hspace{.1in} \includegraphics[scale=0.47]{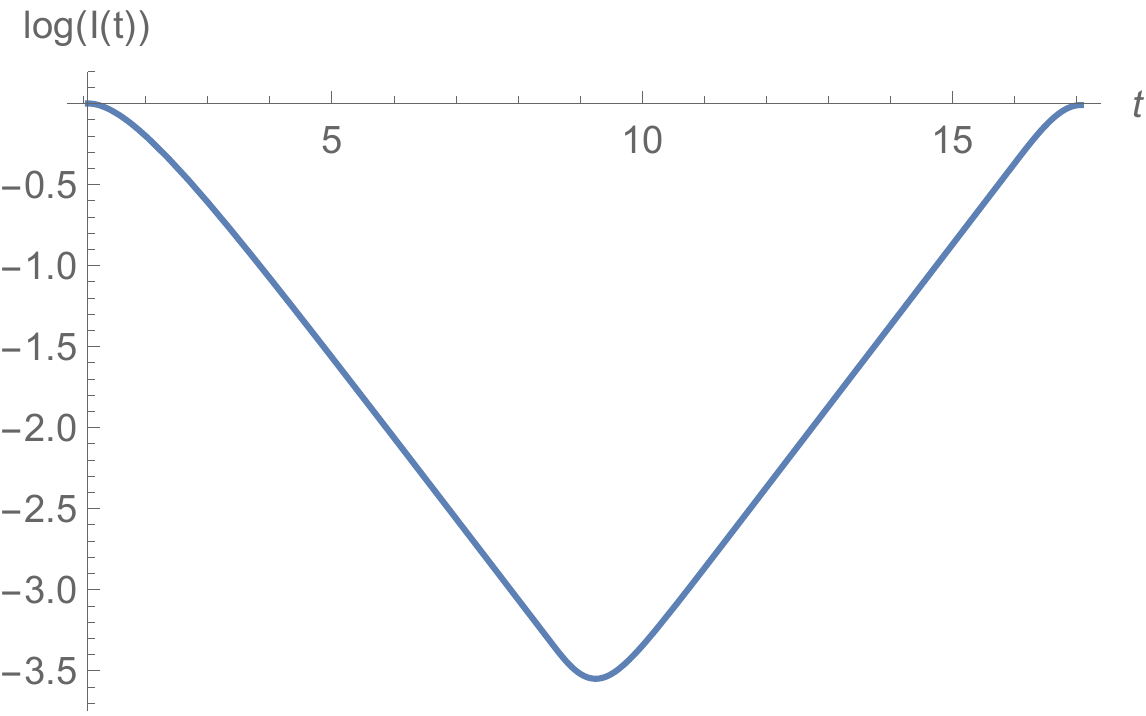}\hspace{.1in}\includegraphics[scale=0.47
]{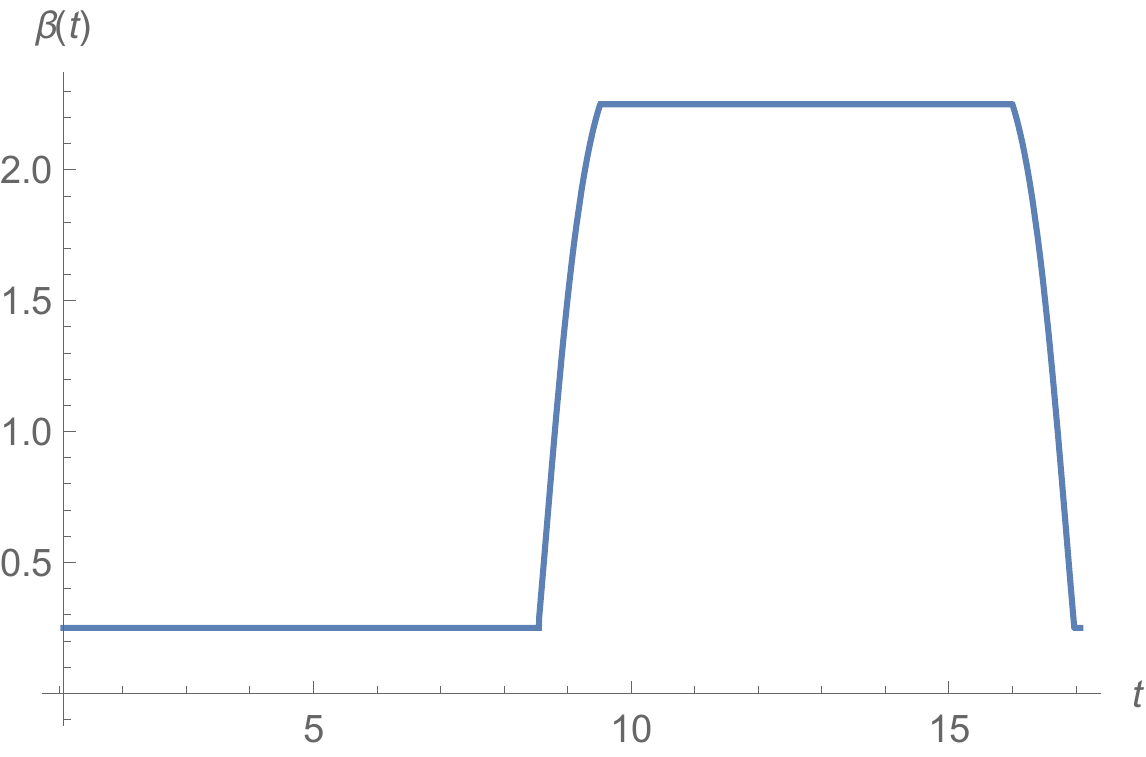}
\caption{\label{fig::usefulexamplechart} Analog of (one period of) Figure~\ref{fig::examplechart} with $u(x) = x-x^2/10$. Similar to Figure~\ref{fig::examplechart} except that the $\beta$ transitions (right) from $\bmin$ to $\bmax$ and back take place gradually over a couple of time units instead of instantly. Here the corners in the graph of $V(T)$ are less sharp than in Figure~\ref{fig::examplechart}, and the turnarounds in $\log I(t)$ are smoother. One might expect these effects to be more pronounced if $u$ were more concave. Up to affine transformation, the $u$ here is very close to the example in \eqref{eqn::ub} with $a_1=3/4$ and $a_2=1/4$.}
\end{center}
\end{figure}

\subsection{Euler-Lagrange solutions}

In principle, one can find an optimal $V$ explicitly using Euler-Lagrange theory. We briefly sketch the idea here. Consider an interval $(s, s+\Delta)$ on which $V$ is known to increase monotonically from $x_1$ to $x_2$ and assume $\int_s^{s+\Delta}V(t)dt = \Lambda$.  Once $\Lambda$ and $\Delta$ are fixed, one can fix any constants $a$ and $b$ and aim to maximize 
 \begin{equation}\label{eqn::phiw}\int_s^{s+\Delta}  w\bigl( V(t) \bigr) dt + \int_s^T  G_{V(t)}\bigl(\dot V(t)\bigr) dt\end{equation}
where $w(v):=u \circ \phi(v) + av + b$. The extra $av+b$ terms do not affect the optimal solution, since the amount that they add to \eqref{eqn::phiw} is determined by $\Lambda$ and $\Delta$.  However, one can also let $\Lambda$ and $\Delta$ be {\em variable} parameters and then try to tune $a$ and $b$ so that the optimizer to \eqref{eqn::phiw} obtains the desired values.

During the interval $(t,t+\Delta)$ we interpret $h(x) := \dot V(V^{-1}(x))$ as the ``speed'' at which $x$ is passed through, for $x \in (x_1, x_2)$, so that $1/h(x)$ is the density function for the occupation measure at $x$, and the goal becomes to maximize $\int_{x_1}^{x_2} \frac{G_x( h(x)) + w(x)}{h(x)}dx$.  We can then use calculus to find (for each $x$) the $h(x)$ that optimizes the integrand, recalling the constraints on $h(x)$ from \eqref{eqn::Yconstraint}. If the optimizer is unique for each $x$, this determines the function $h$. Once $h$ is known, solving the ODE $\dot V(t) = h(V(t))$ allows us to produce analogs of the red curves in Figure~\ref{fig::ychart} that dictate the way $V$ evolves during its upward trajectories.  We can treat decreasing intervals similarly, obtaining analogs of the blue curves in Figure~\ref{fig::ychart}.

In general, finding $a$ and $b$ is a tricky optimization problem; however, if we {\em assume} or {\em guess} that the optimal $V$ has an interval on which $V(t)$ is close to $\vmin$ (and ergo $\dot V(t) \approx 0$) and an interval on which $V(t)$ is close to $\vmax $, then we can deduce that $w$ must be close to zero at $\vmin$ and $\vmax$ (since otherwise one could increase \eqref{eqn::phiw} by either prolonging or condensing these intervals) which determines approximately what $a$ and $b$ must be. (It is not hard to see that---if the mean of $V$ and its endpoints in $(\vmin, \vmax)$ are held fixed---this assumption is correct if $T$ is large enough, but incorrect for smaller $T$.) Once $a$ and $b$ are known---and analogs of the blue and red curves in Figure~\ref{fig::ychart} are drawn---the problem of figuring out where the ``turnarounds'' occur is essentially the same here as in Proposition~\ref{prop::bestpossible}. This approach was used to produce Figure~\ref{fig::usefulexamplechart}.

Although we will not give details, we expect the arguments in the proof of Proposition~\ref{prop::wallconstrained} to work in a general $u$ version of Proposition~\ref{prop::bestwithu}, enabling one to show that the long-term optimal $\log I$ oscillates between $C_1$ and $C_2$ in a similar fashion (with the blue and red curves coming from {\em some} choice of $a$ and $b$). If the $a$ and $b$ are different from the ones guessed in producing Figure~\ref{fig::examplechart}, then the shape of the turnarounds might be different as well.

\end{appendix}

\bibliography{moderateseir}

\bibliographystyle{alpha}

\end{document}